\documentclass[final]{article}

\usepackage{graphicx}
\usepackage{amsmath}
\usepackage{amsfonts}
\usepackage{amssymb}
 \usepackage{amscd}

\newtheorem{theorem}{Theorem}

\newtheorem{example}{Example}

\newtheorem{lemma}[theorem]{Lemma}

\newenvironment{proof}[1][Proof]{\textbf{#1.} }{\ \rule{0.5em}{0.5em}}

\begin{document}

\title{ On an infinite number of solutions to the  Diophantine equation  $ x^{n}+y^{p}=z^{q}$
 over the square integer matrices }
\author{By I. Kaddoura and B. Mourad }

\author{Issam Kaddoura,$^{\rm a}$   Bassam Mourad $^{\rm b}$  \\  $^{\rm a}${\em{ Faculty  of Arts and Science, Lebanese International
University, Saida, Lebanon }}\\ $^{\rm b}${\em{Department of Mathematics,  Faculty of Science, Lebanese University, Beirut, Lebanon}}}

\maketitle

\begin{abstract} In this paper, we use some extension of the Cayley-Hamilton
theorem to find  a family of matrices with integer entries that satisfy the non-linear Diophantine equation  $ x^{n}+y^{p}=z^{q}$
where $n,p$  and $q$ are arbitrary positive integers.
\end{abstract}

\section*{Main results} In \cite{km} we have proved the following  extension of the
Cayley-Hamilton theorem.
\begin{theorem}  Let $f(x,y)=$ $\sum_{k=0}^{n}a_{k}x^{k}y^{n-k}$  be any homogeneous polynomial in the two
variable x and y where all the
a$_{k}$ are complex and such that $a_{n}\neq0.$  Let $A$ be the companion matrix of $f(x,y)$, then for all complex numbers $r$ and $s$, we have
$
f(rA+s(\det A)I_{n},rI_{n}+sadjA)=0
$
 where $I_{n}$ is the identity and $adjA$
is the adjoint matrix of A.
\end{theorem}

For our purposes, we need the following well-known results concerning polynomials with integer coefficients.
\begin{lemma}  For any positive integer $k,$  the following relations hold:
\[
(x+y)^{6k+1}-x^{6k+1}-y^{6k+1}\equiv0\operatorname{mod}\allowbreak\left(
xy+x^{2}+y^{2}\right)  \ \ \ \ \ \ \ (1)
\]
\[
(x+y)^{6k+5}-x^{6k+5}-y^{6k+5}=0\operatorname{mod}\allowbreak\allowbreak
\left(  xy+x^{2}+y^{2}\right)  ^{2}\ \ \ \ \ \ (2)
\]
\[
\frac{(x+y)^{6k+1}-x^{6k+1}-y^{6k+1}}{xy+x^{2}+y^{2}}\equiv0\operatorname{mod}%
\allowbreak f(x,y)\ \ \ \ \ \ \ \ \ \ \ \ \ \ \ \ \ \ (3)
\]
\[
\frac{(x+y)^{6k+5}-x^{6k+5}-y^{6k+5}}{xy+x^{2}+y^{2}}\equiv0\operatorname{mod}%
\allowbreak g(x,y)\ \ \ \ \ \ \ \ \ \ \ \ \ \ \ \ \ \ (4)
\]
\[
\frac{(x+y)^{6k+5}-x^{6k+5}-y^{6k+5}}{\left(  xy+x^{2}+y^{2}\right)  ^{2}%
}\equiv0\operatorname{mod}\allowbreak
h(x,y)\ \ \ \ \ \ \ \ \ \ \ \ \ \ \ \ \ \ (5)
\]
where $\allowbreak f(x,y)$, $g(x,y)$, and $ h(x,y)$ are polynomials
with integer coefficients of order $6k-1,6k+3,$ and $6k+1$ respectively.
\end{lemma}
Now using the preceding lemma and theorem, we prove the following main result.
\begin{theorem} The non-linear Diophantine equation  $ x^{n}+y^{p}=z^{q}$
where $n,p$  and $q$ are arbitrary positive integers, have an infinite
number of non trivial solutions over the square integer matrices.
\end{theorem} 
\begin{proof} Let $f(x,y)=xy+x^{2}+y^{2}$ and using the above theorem,
we obtain%
\begin{align*}
f(rA+(\det A)I,rI+sadjA)  & =f(r\left[
\begin{array}
[c]{cc}%
0 & -1\\
1 & -1
\end{array}
\right]  +sI,rI+s\left[
\begin{array}
[c]{cc}%
-1 & 1\\
-1 & 0
\end{array}
\right]  )\\
& =f(\left[
\begin{array}
[c]{cc}%
s & -r\\
r & -r+s
\end{array}
\right]  \allowbreak,\left[
\begin{array}
[c]{cc}%
r-s & s\\
-s & r
\end{array}
\right]  )=\allowbreak\left[
\begin{array}
[c]{cc}%
0 & 0\\
0 & 0
\end{array}
\right].
\end{align*}
Now using identities (1) and (2) in the preceding lemma, we get:
\[
\left[
\begin{array}
[c]{cc}%
r & -r+s\\
r-s & s
\end{array}
\right]  ^{6k+1}-\left[
\begin{array}
[c]{cc}%
s & -r\\
r & -r+s
\end{array}
\right]  ^{6k+1}-\left[
\begin{array}
[c]{cc}%
r-s & s\\
-s & r
\end{array}
\right]  ^{6k+1}=0\ \text{for all }k\in \mathbb{N}\text{ and }r,s\in \mathbb{Z}.
\]
Similarly, we have
\[
\left[
\begin{array}
[c]{cc}%
r & -r+s\\
r-s & s
\end{array}
\right]  ^{6k+5}-\left[
\begin{array}
[c]{cc}%
s & -r\\
r & -r+s
\end{array}
\right]  ^{6k+5}-\left[
\begin{array}
[c]{cc}%
r-s & s\\
-s & r
\end{array}
\right]  ^{6k+5}=0\text{ for all }k\in \mathbb{N}\text{ and }r,s\in \mathbb{Z}.
\]
By choosing $r$ and $s$ any arbitrary integers, we obtain an infinite family of $2\times 2$
 integer matrices which are solutions to the
equation $\ A^{n}+B^{p}=C^{q}$ where $n, p$ and $q$ are odd integers such that 
$npq$ has the form $6k+1$ or $6k+5.$  Then in this case,
$x^{npq}+y^{npq}=z^{npq}$ can be written as $(x^{np})^{q}+(y^{pq})^{n}%
=(z^{nq})^{p}$ to obtain the required solution.
Next, notice that by applying the preceding theorem and taking into account (3), (4) and (5) of the preceding lemma, we can extend
the family of solutions into matrices of sizes $(6k-1)\times(6k-1),(6k+3)\times(6k+3),$ and
$(6k+1)\times(6k+1)$ respectively.
\end{proof}

\begin{example}
 Consider the following examples.
\begin{enumerate}
\item $ \left[
\begin{array}
[c]{cc}%
s & -r\\
r & -r+s
\end{array}
\right]  ^{31}+\left[
\begin{array}
[c]{cc}%
r-s & s\\
-s & r
\end{array}
\right]  ^{31}\allowbreak\allowbreak=\left[
\begin{array}
[c]{cc}%
r & -r+s\\
r-s & s
\end{array}
\right]  ^{31}$

\item $ \left[
\begin{array}
[c]{cc}%
-54 & 27\\
-27 & -27
\end{array}
\right]  ^{11}+\left[
\begin{array}
[c]{cc}%
-1 & 2\\
-2 & 1
\end{array}
\right]  ^{77}=\left[
\begin{array}
[c]{cc}%
486 & -243\\
243 & 243
\end{array}
\right]  ^{7}$
\item $\left[
\begin{array}
[c]{cc}%
62 & -149\\
149 & -87
\end{array}
\right]  ^{10}+\left[
\begin{array}
[c]{cc}%
-3 & 8\\
-8 & 5
\end{array}
\right]  ^{25}=\ \ \allowbreak\left[
\begin{array}
[c]{cc}%
-14\,632 & 18\,357\\
-18\,357 & 3725
\end{array}
\right]  ^{5}.$
\end{enumerate}
\end{example}

\end{document}